\documentclass[smallextended]{svjour3}
\usepackage{amsmath,amssymb}
\usepackage{graphicx}

\spnewtheorem*{xproof}{}{\itshape}{\rmfamily}

\renewenvironment{proof}[1][\proofname]
 {\renewcommand\xproofname{#1}\xproof}
 {\endxproof}

\journalname{Graphs and Combinatorics}
\begin{document}
\title{The Hamilton-Waterloo problem for triangle-factors and heptagon-factors}
\author{Hongchuan Lei\and {Hung-Lin Fu}}
\institute{Hongchuan Lei\at Department of Applied Mathematics, National Chiao Tung University, Hsinchu 30010, Taiwan; \at Institute of Mathematics, Academia Sinica, Taipei 10617, Taiwan  \\
\email{hongchuanlei@gmail.com}
\and Hung-Lin Fu\at Department of Applied Mathematics, National Chiao Tung University, Hsinchu 30010, Taiwan
 \\\email{hlfu@math.nctu.edu.tw}}
\date{Received: date / Accepted: date}
\maketitle
\begin{abstract}
Given 2-factors $R$ and $S$ of order $n$, let $r$ and $s$ be nonnegative integers with $r+s=\lfloor \frac{n-1}{2}\rfloor$, the Hamilton-Waterloo problem asks for a 2-factorization of $K_n$ if $n$ is odd, or of $K_n-I$ if $n$ is even, in which $r$ of its 2-factors are isomorphic to $R$ and the other $s$ 2-factors are isomorphic to $S$. In this paper, we solve the problem for the case of triangle-factors and heptagon-factors for odd $n$ with 3 possible exceptions when $n=21$.
\end{abstract}
\section{Introduction}
A decomposition of a graph $G$ is a collection of edge-disjoint subgraphs such that every edge of $G$ belongs to exactly one of the subgraphs. A subgraph $F$ of a graph $G$ is a factor if $F$ contains all the vertices of $G$, if each component of $F$ is isomorphic to a graph $H$, then $F$ is called an $H$-factor of $G$, while if $F$ is a $d$-regular graph, then we call $F$ a $d$-factor.  A $C_k$-factor is a 2-factor consisting entirely of cycles of length $k$. A factorization of a graph $G$ is a decomposition of $G$ such that each subgraph is a factor, if the factors are all 2-factors then it is called a 2-factorization. An $\{H_1^{m_1},H_2^{m_2},\ldots,H_t^{m_t}\}$-factorization of a graph $G$ is a factorization of $G$ in which there are precisely $m_i$ $H_i$-factors. If such a factorization exists, we say that $(G;H_1^{m_1},H_2^{m_2},\ldots,H_t^{m_t})$ exists.

Given 2-factors $R$ and $S$ of oder $n$, let $r$ and $s$ be nonnegative integers with $r+s=\lfloor \frac{n-1}{2}\rfloor$, the Hamilton-Waterloo problem asks for a 2-factorization of the complete graph $K_n$ if $n$ is odd, or $K_n-I$ if $n$ is even, in which $r$ of its 2-factors are isomorphic to $R$ and the other $s$ 2-factors are isomorphic to $S$, where $I$ is a 1-factor. The goal of the problem is to determine the spectrum of $r$ (or $s$) for all possible $n$, i.e. the set of $r$ (or $s$) such that the corresponding 2-factorization of $K_n$ or $k_n-I$ exists. If $R$ is a $C_m$-factor and $S$ is a $C_k$-factor, i.e. each 2-factor is uniform, then such a 2-factorization is denoted by $HW(n;r,s;m,k)$.

The uniform cases of the Hamilton-Waterloo problem have attracted much attention in the last decade. The existence of $HW(n;r,s;m,k)$ has been settled when $r=0$ or $s=0$ in \cite{Alspach1,Alspach2,Hoffman}. So we only discuss the case $rs\neq 0$ in this paper.

\begin{theorem}\cite{Alspach1,Alspach2,Hoffman}\label{1.0}
Let $n\geq 3$ and $m\geq 3$. Let $G=K_n$ if $n$ is odd, $G=K_n-I$ if $n$ is even. Then $(G;C_m^{\lfloor \frac{n-1}{2}\rfloor})$ exists if and only if $n\equiv 0\pmod{m}$ and $(n,m)\notin \{(6,3),(12,3)\}$.
\end{theorem}

Adams et al. \cite{Adams} dealt with the cases $(m,k)\in\{(4,6),$ $(4,8),$ $(4,16),$ $(8,16),$ $(3,5),$ $(3,15),$ $(5,15)\}$ and completely solved some of them, they also introduced some methods. Danziger et al.\cite{Danziger} almost completely solved the case $(m,k)=(3,4)$ with only 9 possible exceptions. The case $(m,k)=(n,3)$, i.e. $R$ is a Hamilton cycle and $S$ is a triangle-factor, was studied in \cite{Dinitz1,Dinitz2,Horak,Lei2}, and is still open. In recent years, remarkable progress has been made on the Hamilton-Waterloo problem when both $R$ and $S$ are consist of even cycles, see \cite{Bryant1,Bryant2,Fu,Lei1}.

The next two lemmas are useful for our constructions, and have been used in many papers, for example see \cite{Fu}.
\begin{lemma}\label{1.2}
Suppose $G_1$ and $G_2$ are two vertex-disjoint graphs. If $(G_1;C_m^r,C_k^s)$ and $(G_2;C_m^r,C_k^s)$ both exist, then  $(G_1\cup G_2;C_m^r,C_k^s)$ exists.
\end{lemma}
\begin{lemma}\label{1.3}
Suppose $G_1$ and $G_2$ are two edge-disjoint graphs with the same vertex set. If $(G_1;C_m^{r_1},C_k^{s_1})$ and $(G_2;C_m^{r_2},C_k^{s_2})$ both exist, then $(G_1\cup G_2;C_m^{r_1+r_2},\\C_k^{s_1+s_2})$ exists.
\end{lemma}
In this paper, we deal with the case $(m,k)=(3,7)$ with $n$ odd. Lemma 3.3 in \cite{Adams} shows that if $HW(21;r,s;3,7)$ exists for all nonnegative integers $r$ and $s$ with $r+s=10$ then the problem is settled. Unfortunately we can't construct all possible 2-factorizations of this kind for $n=21$. Instead, using 2-factorizations of $K_{7,7,7}$, we will prove the following result.
\begin{theorem}\label{1.1}
If $n\equiv 1\pmod 2$ and $rs\neq 0$ with $r+s=\frac{n-1}{2}$, then there exists an $HW(n;r,s;3,7)$ if and only if $n\equiv {21}\pmod{42}$ except possibly when $n=21$ and $r=2,4,6$.
\end{theorem}
In Section 2, we decompose $K_{7,7,7}$ into $C_3$-factors and $C_7$-factors. In Section 3, we deal with $K_{21}$. In Section 4, we show how to decompose $K_n$ into $K_{7,7,7}$-factors and $K_{21}$-factors, then prove Theorem \ref{1.1}.

\section{Factorizations of $K_{7,7,7}$}
Let $V(K_{7,7,7})=\{j_i\mid j\in Z_7,i\in Z_3\}$, and let $V_i=\{j_i\mid j\in Z_7\}$ for $i\in Z_3$ be the three partite sets of $K_{7,7,7}$. Denote the complete graph on $V_i$ by $K_{V_i}$, the complete bipartite graph on $V_i$ and $V_j$ by $K_{V_i,V_j}$, and the complete tripartite graph $K_{7,7,7}$ on $V_0$, $V_1$ and $V_2$ by $K_{V_0,V_1,V_2}$. Then $$E(K_{V_0,V_1,V_2})=E(K_{V_0,V_1})\cup E(K_{V_1,V_2})\cup E(K_{V_2,V_0}).$$
For $i,j\in Z_3$ and $d\in Z_7$, let $E_{ij}(d)=\{\{l_i,(l+d)_j\}\mid l\in Z_7\}$. It is easy to verify that
$$E(K_{V_i})=\bigcup\limits_{d=1}^{3}{E_{ii}(d)},$$
$$E(K_{V_i,V_j})=\bigcup\limits_{d=0}^{6}{E_{ij}(d)}\ \text{for}\  i\neq j.$$

Some of the techniques used in the following lemmas are widely used in combinatorial designs, see \cite{Liu} for example. In the beginning we give a few basic constructions. The first two lemmas are easy to see, so we omit the proofs.
\begin{lemma}\label{2.1}
Let $d_0,d_1,d_2\in Z_7$. If $d_0+d_1+d_2\equiv 0\pmod 7$, then the edges of $E_{01}(d_0)\cup E_{12}(d_1)\cup E_{20}(d_2)$ form a $C_3$-factor of $K_{7,7,7}$.
\end{lemma}

\begin{lemma}\label{2.2}
If $(d,7)=1$, then the edges of $E_{ii}(d)$ form a Hamilton cycle, i.e. a $C_7$-factor of $K_{V_i}$.
\end{lemma}

\begin{lemma}\label{2.3}
The edges of $\bigcup\limits_{d\in\{1,6\}}^{}{(E_{01}(d)\cup E_{12}(d)\cup E_{20}(d))}$ can be decomposed into 2 $C_7$-factors of $K_{7,7,7}$.
\end{lemma}
\begin{proof}
Let
\begin{align*}
F_1 & =\{(0_i,1_{i+1},2_{i+2},3_{i},4_{i+1},5_i,6_{i+1})\mid i\in Z_3\},\\
F_2 & =\{(0_i,1_{i+2},2_{i+1},3_{i},4_{i+2},5_i,6_{i+2})\mid i\in Z_3\},
\end{align*}
then both $F_1$ and $F_2$ are $C_7$-factors of $K_{7,7,7}$. It is straightforward to verify that
$$E(F_1)\cup E(F_2)=\bigcup\limits_{d\in\{1,6\}}^{}{(E_{01}(d)\cup E_{12}(d)\cup E_{20}(d))}.$$
\end{proof}

\begin{lemma}\label{2.4}
The edges of $\bigcup\limits_{d\in\{2,5\}}^{}{(E_{01}(d)\cup E_{12}(d)\cup E_{20}(d))}$ can be decomposed into 2 $C_7$-factors of $K_{7,7,7}$.
\end{lemma}
\begin{proof}
The proof is similar to Lemma \ref{2.3}, let the 2 $C_7$-factors be
\begin{align*}
F_1 & =\{(0_i,2_{i+1},4_{i+2},6_{i},1_{i+1},3_i,5_{i+1})\mid i\in Z_3\},\\
F_2 & =\{(0_i,2_{i+2},4_{i+1},6_{i},1_{i+2},3_i,5_{i+2})\mid i\in Z_3\}.
\end{align*}
\end{proof}

\begin{lemma}\label{2.5}
The edges of $\bigcup\limits_{d\in\{3,4\}}^{}{(E_{01}(d)\cup E_{12}(d)\cup E_{20}(d))}$ can be decomposed into 2 $C_7$-factors of $K_{7,7,7}$.
\end{lemma}
\begin{proof}
Let the 2 $C_7$-factors be
\begin{align*}
F_1 & =\{(0_i,3_{i+1},6_{i+2},2_{i},5_{i+1},1_i,4_{i+1})\mid i\in Z_3\},\\
F_2 & =\{(0_i,3_{i+2},6_{i+1},2_{i},5_{i+2},1_i,4_{i+2})\mid i\in Z_3\}.
\end{align*}
\end{proof}
\begin{lemma}\cite{Alspach2}\label{2.6}
Let $K_{d(m)}$ be the complete multipartite graph with $d$ parts of size $m$, if $d$ and $m$ are both odd integers, then there is a 2-factorization of $K_{d(m)}$, in which each 2-factor is a $C_m$-factor.
\end{lemma}
Now we decompose $K_{7,7,7}$ into $C_3$-factors and $C_7$-factors.
\begin{lemma}\label{2.7}
$(K_{7,7,7};C_3^\alpha,C_7^\beta)$ exists for $\alpha\in\{0,1,3,5,7\}$ with $\alpha+\beta=7$.
\end{lemma}
\begin{proof}
For $\alpha=0$, $(K_{7,7,7};C_7^7)$ exists by Lemma \ref{2.6}.

For $\alpha=1$, decompose $\{E_{01}(d)\cup E_{12}(d)\cup E_{20}(d)\mid d=1,2,\ldots,6\}$ into 6 $C_7$-factors by Lemma \ref{2.3}-\ref{2.5}, the remaining edges $E_{01}(0)\cup E_{12}(0)\cup E_{20}(0)$ form a $C_3$-factor by Lemma \ref{2.1}.

For $\alpha=3$, decompose $\{E_{01}(d)\cup E_{12}(d)\cup E_{20}(d)\mid d=2,3,4,5\}$ into 4 $C_7$-factors by Lemma \ref{2.4} and Lemma \ref{2.5}. The 3 $C_3$-factors are $E_{i(i+1)}(0)\cup E_{(i+1)(i+2)}(1)\cup E_{(i+2)i}(6), i\in Z_3$ by Lemma \ref{2.1}.

For $\alpha=5$, decompose $\{E_{01}(d)\cup E_{12}(d)\cup E_{20}(d)\mid d=3,4\}$ into 2 $C_7$-factors by Lemma \ref{2.5}. By Lemma \ref{2.1} the 5 $C_3$-factors are
\begin{align*}
& E_{01}(0)\cup E_{12}(1)\cup E_{23}(6),\ E_{01}(2)\cup E_{12}(0)\cup E_{23}(5), & \\
& E_{01}(5)\cup E_{12}(2)\cup E_{23}(0),\ E_{01}(6)\cup E_{12}(6)\cup E_{23}(2), & \\
& E_{01}(1)\cup E_{12}(5)\cup E_{23}(1).
\end{align*}

For $\alpha=7$, by Lemma \ref{2.1} the 7 $C_3$-factors are
\begin{align*}
& E_{i(i+1)}(1)\cup E_{(i+1)(i+2)}(2)\cup E_{(i+2)i}(4),\ i\in Z_3;\\
& E_{j(j+1)}(3)\cup E_{(j+1)(j+2)}(5)\cup E_{(j+2)j}(6),\ j\in Z_3;\\
& E_{01}(0)\cup E_{12}(0)\cup E_{23}(0).
\end{align*}
\end{proof}
\section{Factorizations of $K_{21}$}
In this section, $V_i, K_{V_i}, K_{V_i,V_j}, K_{V_0,V_1,V_2},$ and $E_{ij}(d)$ have the same meanings as given in Section 2. Let $V(K_{21})=V_0\cup V_1\cup V_2$, then the edge set $$E(K_{21})=\left(\bigcup_{i\in Z_3}^{}E(K_{V_i})\right)\cup E(K_{V_0,V_1,V_2}).$$
Now we decompose $K_{21}$ into $\gamma$ $C_3$-factors and $\delta$ $C_7$-factors with $\gamma+\delta=10$.
\begin{lemma}\label{3.1}
$(K_{21};C_3^\gamma,C_7^\delta)$ exists for $\gamma\in\{0,1,3,5,7,10\}$ with $\gamma+\delta=10$.
\end{lemma}
\begin{proof}
Since $E(K_{V_i})=\bigcup\limits_{d=1}^{3}{E_{ii}(d)}$ for $i\in Z_3$, by Lemma \ref{2.2}, $(K_{V_i};C_7^3)$ exists. Then by Lemma \ref{1.2}, $(\bigcup_{i\in Z_3}^{}K_{V_i};C_7^3)$ exists. Hence, it is easy to observe that if $(K_{V_0,V_1,V_2};C_3^\alpha,C_7^\beta)$ exists, then $(K_{V_1\cup V_2\cup V_3};C_3^\alpha,C_7^{\beta+3})$ exists by Lemma \ref{1.3}. Thus by Lemma \ref{2.7}, $(K_{21};C_3^\gamma,C_7^\delta)$ exists for $\gamma\in\{0,1,3,5,7\}$ with $\gamma+\delta=10$.

For $(\gamma,\delta)=(10,0)$, $(K_{21};C_3^\gamma,C_7^\delta)$ exists by Theorem \ref{1.0}.
\end{proof}

\begin{lemma}\label{3.2}
$(K_{21};C_3^\gamma,C_7^\delta)$ exists for $(\gamma,\delta)=(8,2)$.
\end{lemma}
\begin{proof}
Let \begin{align*}
F_0=&\{(0_0,1_0,2_1),(1_1,4_1,5_2),(1_2,6_2,3_0),(2_0,6_1,4_2),\\
& (4_0,0_1,3_2),(5_0,3_1,2_2),(6_0,5_1,0_2)\},
\end{align*}
then $F_0$ is a $C_3$-factor of $K_{21}$. Six additional $C_3$-factors, denoted by $F_1,F_2,\ldots,$ $F_6$, are formed by developing $F_0$ mod$(7,-)$.
Let $F_7=E_{01}(0)\cup E_{12}(0)\cup E_{20}(0)$, then by Lemma \ref{2.1} $F_7$ is a $C_3$-factor.
Let $F_8=E_{00}(2)\cup E_{11}(1)\cup E_{22}(1)$, $F_9=E_{00}(3)\cup E_{11}(2)\cup E_{22}(3)$, then $F_8$ and $F_9$ are both $C_7$-factors of $K_{21}$ by Lemma \ref{1.2} and Lemma \ref{2.2}. Finally, one can check that each edge of $K_{21}$ is used exactly once.
\end{proof}

\begin{lemma}\label{3.3}
$(K_{21};C_3^\gamma,C_7^\delta)$ exists for $(\gamma,\delta)=(9,1)$.
\end{lemma}
\begin{proof}
Let \begin{align*}
F_0=&\{(0_0,1_0,6_1),(0_1,1_1,4_2),(0_2,2_2,3_0),(2_0,4_0,4_1),\\
& (3_1,5_1,3_2),(5_2,6_2,5_0),(6_0,2_1,1_2)\},
\end{align*}
then $F_0$ is a $C_3$-factor of $K_{21}$. Six additional $C_3$-factors, denoted by $F_1,F_2,\ldots,$ $F_6$, are formed by developing $F_0$ mod$(7,-)$.
Let $F_7=E_{01}(1)\cup E_{12}(2)\cup E_{20}(4)$, $F_8=E_{01}(4)\cup E_{12}(1)\cup E_{20}(2)$, then by Lemma \ref{2.1} $F_7$ and $F_8$ are both $C_3$-factors.
Let $F_9=E_{00}(3)\cup E_{11}(3)\cup E_{22}(3)$, then $F_9$ is a $C_7$-factor of $K_{21}$ by Lemma \ref{1.2} and Lemma \ref{2.2}. Again, one can check that each edge of $K_{21}$ is used exactly once.
\end{proof}
Combining Lemmas \ref{3.1}-\ref{3.3}, we have the following result.
\begin{lemma}\label{3.4}
$(K_{21};C_3^\gamma,C_7^\delta)$ exists for $\gamma\in\{0,1,3,5,7,8,9,10\}$ with $\gamma+\delta=10$.
\end{lemma}
\section{Main Results}
Let $n$ be an odd integer. Let $r$ and $s$ be positive integers with $r+s=\frac{n-1}{2}$. It is easy to see that a necessary condition for the existence of an $HW(n;r,s;3,7)$ is $n\equiv {21} \pmod{42}$. Let $n=42t+21$, $t\geq 0$. Let the vertex set of $K_n$ be $V(K_n)=\{j_i\mid j\in Z_7, i\in Z_{6t+3}\}$, denote $V_i=Z_7\times\{i\}$ for $i\in Z_{6t+3}$. The next lemma is based on a construction given in the paper \cite{Adams}.
\begin{lemma}\label{4.1}
For $n=42t+21$ and $t\geq 0$, $(K_n;K_{7,7,7}^{3t},K_{21})$ exists.
\end{lemma}
\begin{proof}
By Theorem \ref{1.0}, $(K_{6t+3};C_3^{3t+1})$ exists for $t\geq 0$, it is actually the well known Kirkman triple system of order $6t+3$. Let the vertex set of $K_{6t+3}$ be $\{V_i\mid i\in Z_{6t+3}\}$, replace each 3-cycle $(V_i,V_j,V_k)$ with the complete tripartite graph $K_{7,7,7}$ on vertex sets $V_i,V_j$ and $V_k$, then each $C_3$-factor of $K_{6t+3}$ corresponds to a $K_{7,7,7}$-factor of $K_n$, also these $K_{7,7,7}$-factors form the complete multipartite graph $K_{(6t+3)(7)}$ on vertex sets $V_0,V_1,\ldots,V_{6t+2}$, i.e. $(K_{(6t+3)(7)};K_{7,7,7}^{3t+1})$ exists. Hence $(K_{n};K_{7,7,7}^{3t+1},K_7)$ exists and the union of any $K_{7,7,7}$-factor and the $K_7$-factor of $K_n$ is actually a $K_{21}$-factor. Therefore, $(K_{n};K_{7,7,7}^{3t},$ $K_{21})$ exists.
\end{proof}

\begin{lemma}\label{4.2}
Let $\alpha_i\in\{0,1,3,5,7\}$ with $\alpha_i+\beta_i=7$ for $i=1,2,\ldots,3t$, and $\gamma\in\{0,1,3,5,7,8,9,10\}$ with $\gamma+\delta=10$, then there exists an $HW(n;\sum_{i=1}^{3t}\alpha_i+\gamma,\sum_{i=1}^{3t}\beta_i+\delta;3,7)$.
\end{lemma}
\begin{proof}
By Lemma \ref{4.1}, we decompose $K_n$ into $3t$ $K_{7,7,7}$-factors and a $K_{21}$-factor.

For the $i$th $K_{7,7,7}$-factor, let $\alpha_i\in\{0,1,3,5,7\}$ and $\alpha_i+\beta_i=7$. Then decompose each $K_{7,7,7}$ of this $K_{7,7,7}$-factor into $\alpha_i$ $C_3$-factors and $\beta_i$ $C_7$-factors by Lemma \ref{2.7}, by Lemma \ref{1.2} these 2-factors of $K_{7,7,7}$ form $\alpha_i$ $C_3$-factors and $\beta_i$ $C_7$-factors of $K_n$.

Similarly, the $K_{21}$-factor of $K_n$ can be decomposed into $\gamma$ $C_3$-factors and $\delta$ $C_7$-factors for $\gamma\in\{0,1,3,5,7,8,9,10\}$ with $\gamma+\delta=10$ by Lemma \ref{1.2} and \ref{3.4}.

Then by Lemma \ref{1.1}, $(K_n;C_3^{\sum_{i=1}^{3t}\alpha_i+\gamma},C_7^{\sum_{i=1}^{3t}\beta_i+\delta})$ exists, i.e. there exists an $HW(n;$ $\sum_{i=1}^{3t}\alpha_i+\gamma,\sum_{i=1}^{3t}\beta_i+\delta;3,7)$.
\end{proof}
We are now ready to prove the main theorem of this paper.
\begin{proof}[{Proof of Theorem~\ref{1.1}}]
As noted earlier, the condition $n\equiv 21\pmod {42}$ is necessary, we now prove sufficiency. Let $n=42t+21$, the case $t=0$ (i.e. $n=21$) is solved by Lemma \ref{3.4}.

For the case $t>0$, let $r=7a+b$, where $0\leq b<7$. For the existence of an $HW(n;r,s;3,7)$, we only need to assign a proper value to each of $\{\gamma,\alpha_i\mid i=1,2,\ldots,3t\}$ in Lemma \ref{4.2}. Note that if $a=3t+1$, then $b<3$ (the case $b=3$ is the case $s=0$, which is covered by Theorem~\ref{1.0}).

If $b=0$ and $a<3t+1$, then let $\gamma=0$ and $\alpha_i=\begin{cases}7, & \text{for}\ 1\leq i\leq a,\\0, & \text{for}\ a< i\leq 3t. \end{cases}$

If $b=0$ and $a=3t+1$, then let $\gamma=7$ and $\alpha_i=7$ for $i=1,2,\ldots,3t$.

If $b=1$ and $a<3t+1$, then let $\gamma=1$ and $\alpha_i=\begin{cases}7, & \text{for}\ 1\leq i\leq a,\\0, & \text{for}\ a< i\leq 3t. \end{cases}$

If $b=1$ and $a=3t+1$, then let $\gamma=8$ and $\alpha_i=7$ for $i=1,2,\ldots,3t$.

If $b=2$ and $a<3t$, then let $\gamma=1$ and $\alpha_i=\begin{cases}1, & \text{for}\ i=1,\\7, & \text{for}\ 2\leq i\leq a+1,\\0, & \text{for}\ a+1< i\leq 3t. \end{cases}$

If $b=2$ and $a=3t$, then let $\gamma=9$ and $\alpha_i=\begin{cases}1, & \text{for}\ i=1,\\7, & \text{for}\ 2\leq i\leq 3t. \end{cases}$

If $b=2$ and $a=3t+1$, then let $\gamma=8$ and $\alpha_i=7$ for $i=1,2,\ldots,3t$.

If $b=3$, then let $\gamma=3$ and $\alpha_i=\begin{cases}7, & \text{for}\ \ 1\leq i\leq a,\\0, & \text{for}\ a< i\leq 3t. \end{cases}$

If $b=4$ and $a<3t$, then let $\gamma=3$ and $\alpha_i=\begin{cases}1, & \text{for}\ i=1,\\7, & \text{for}\ 2\leq i\leq a+1,\\0, & \text{for}\ a+1< i\leq 3t. \end{cases}$

If $b=4$ and $a=3t$, then let $\gamma=8$ and $\alpha_i=\begin{cases}3, & \text{for}\ i=1,\\7, & \text{for}\ 1< i\leq 3t. \end{cases}$

If $b=5$, then let $\gamma=5$ and $\alpha_i=\begin{cases}7, & \text{for}\ 1\leq i\leq a,\\0, & \text{for}\ a< i\leq 3t. \end{cases}$

If $b=6$ and $a<3t$, then let $\gamma=1$ and  $\alpha_i=\begin{cases}5, & \text{for}\ i=1,\\7, & \text{for}\ 2\leq i\leq a+1,\\0, & \text{for}\ a+1< i\leq 3t. \end{cases}$

If $b=6$ and $a=3t$, then let $\gamma=8$ and $\alpha_i=\begin{cases}5, & \text{for}\ i=1,\\7, & \text{for}\ 1< i\leq 3t. \end{cases}$
\end{proof}
\section*{Acknowledgments}
The authors would like to express their deep gratefulness to the reviewers for their detail comments and valuable suggestions. The work of Hung-Lin Fu was partially supported by NSC 100-2115-M-009-005-MY3.

\end{document}